\documentclass[11pt,leqno]{article}

\usepackage[active]{srcltx}

\usepackage{kerkis}
\usepackage{amsfonts,latexsym,amsmath,amssymb,amsthm}
\usepackage{dsfont}
\usepackage{fullpage}
\newtheorem{theorem}{Theorem}[section]
\newtheorem{lemma}[theorem]{Lemma}

\newtheorem{corollary}[theorem]{Corollary}
\newtheorem{remark}[theorem]{Remark}

\theoremstyle{definition}

\theoremstyle{remark}

\newtheorem*{note*}{Note}

\numberwithin{equation}{section}

\makeatletter
\newcommand{\rank}{\mathop{\operator@font rank}}
\newcommand{\conv}{\mathop{\operator@font conv}}
\newcommand{\vol}{\mathop{\operator@font vol}}
\newcommand{\onetagright}{\tagsleft@false}
\makeatother

\newcommand{\ls}{\leqslant}
\newcommand{\gr}{\geqslant}
\renewcommand{\epsilon}{\varepsilon}

\usepackage{color}

\usepackage{hyperref}

\begin{document}
\small

\title{\bf Inequalities for the quermassintegrals of sections of convex bodies}

\medskip

\author{Dimitris--Marios Liakopoulos}

\date{}

\maketitle

\begin{abstract}\footnotesize We provide general estimates which compare the quermassintegrals of a convex body
$K$ in ${\mathbb R}^n$ with the averages of the corresponding quermassintegrals of the $k$-codimensional sections
of $K$ over $G_{n,n-k}$. An example is the inequality
$$\alpha_{n,k,j}\frac{W_j(K)}{|K|}\ls\int_{G_{n,n-k}}\frac{W_j(K\cap F)}{|K\cap F|}d\nu_{n,n-k}(F)
\ls \beta_{n,k,j}\frac{W_j(K)}{|K|}$$
where the constants $\alpha_{n,k,j}$ and $\beta_{n,k,j}$ depend only on $n,k$ and $j$, which holds true for any centrally symmetric
convex body $K$ in ${\mathbb R}^n$ and any $0\ls j\ls n-k-1\ls n-1$. Using these estimates we obtain some positive
results for suitable versions of the slicing problem for the quermassintegrals of a convex body.
\end{abstract}

\section{Introduction}\label{section-1}

In this article we provide a number of general double-sided inequalities comparing the quermassintegrals of a convex body
$K$ in ${\mathbb R}^n$ with the averages of the corresponding quermassintegrals of the $k$-codimensional sections
of $K$ over $G_{n,n-k}$. The starting point for studying this type of question is the slicing problem
which asks if there exists an absolute constant $C_1>0$ such that for every $n\gr 2$ and every
convex body $K$ in ${\mathbb R}^n$ with barycenter at the origin (we call these convex bodies centered)
one has
\begin{equation*}|K|^{\frac{n-1}{n}}\ls C_1\,\max_{\xi\in S^{n-1}}\,|K\cap \xi^{\perp }|\end{equation*}
where $|\cdot|$ denotes volume in the appropriate dimension. It is well-known that this problem is equivalent to
the question if there exists an absolute constant $C_2>0$ such that
\begin{equation*}L_n:= \max\{ L_K:K\ \hbox{is an isotropic convex body in}\ {\mathbb R}^n\}\ls C_2\end{equation*}
for all $n\gr 1$, where $L_K$ is the isotropic constant of $K$ (we refer the reader to \cite{BGVV-book} for
background information on isotropic convex bodies). Bourgain proved in \cite{Bourgain-1991} that $L_n\ls c_1\sqrt[4]{n}\log\! n$, and Klartag \cite{Klartag-2006} improved this bound to $L_n\ls c_2\sqrt[4]{n}$. After breakthrough work of Chen \cite{YChen}, Klartag and Lehec have
recently established in \cite{Klartag-Lehec-2022} the polylogarithmic
in the dimension bound $L_n\ls c_3(\ln n)^4$. Even more recently, the method of \cite{Klartag-Lehec-2022} was slightly refined
in \cite{JLV} where it was shown that $L_n\ls c_4(\ln n)^{2.2226}.$ A closer examination of the equivalence of the two questions shows that
\begin{equation*}|K|^{\frac{n-1}{n}}\ls c_3L_n\,\max_{\xi\in S^{n-1}}\,|K\cap \xi^{\perp }|\end{equation*}
for every centered convex body $K$ in ${\mathbb R}^n$. The natural generalization of the problem,
the so-called lower dimensional slicing problem, can be formulated as follows: Let $1\ls k\ls n-1$ and let $\alpha_{n,k}$
be the smallest positive constant $\alpha >0$ such that
\begin{equation*}|K|^{\frac{n-k}{n}}\ls \alpha^k\max_{H\in G_{n,n-k}}|K\cap H|\end{equation*}
for every centered convex body $K$ in ${\mathbb R}^n$, where $G_{n,k}$ is the Grassmann manifold of all $k$-dimensional subspaces of ${\mathbb R}^n$.
Then the question is if there exists an absolute constant $C_4>0$ such that $\alpha_{n,k}\ls C_3$ for all $n$ and $k$.

The slicing problem for surface area, instead of volume, is the question if there exists a constant $\alpha_{n,k}$ such that
\begin{equation}\label{eq:slicing-surface-lower}S(K)\ls\alpha_{n,k}^k|K|^{\frac{k}{n}}\max_{F\in G_{n,n-k}}S(K\cap F)\end{equation}
for every centrally symmetric convex body $K$ in ${\mathbb R}^n$, where $S(A)$ denotes the surface area of
a convex body $A$ in the appropriate dimension. A negative answer to this question was given in \cite{Brazitikos-Liakopoulos-TAMS}.
In fact, this was also done for the natural generalization of the problem in which
surface area is replaced by quermassintegrals of other orders. Recall that if $K$ is a convex body in ${\mathbb R}^n$ then
the function $|K+\lambda B_2^n|$, where $B_2^n$ is the Euclidean unit ball in ${\mathbb R}^n$, is a polynomial in $\lambda\in [0,\infty )$.
We have
\begin{equation*}|K+\lambda B_2^n|=\sum_{j=0}^n \binom{n}{j}W_j(K)\,\lambda^j,\end{equation*}
where $W_j(K)$ is the $j$-th quermassintegral of $K$. This is the classical Steiner formula
(see Section~\ref{section-2} for background information on mixed volumes and quermassintegrals).
The surface area $S(K)$ of $K$ satisfies
\begin{equation*}S(K)=nW_1(K).\end{equation*}
Given $n\gr 2$, $1\ls k\ls n$ and $0\ls j\ls n-k-1$, one may ask if there exists a constant $\alpha_{n,k,j}>0$ such that
\begin{equation}\label{eq:slicing-quermassintegral}W_j(K)\ls\alpha_{n,k,j}^k \,|K|^{\frac{k}{n}}\,\max_{F\in G_{n,n-k}}W_j(K\cap F)\end{equation}
for every centered (or centrally symmetric) convex body $K$ in ${\mathbb R}^n$.
It was proved in \cite{Brazitikos-Liakopoulos-TAMS} that one cannot expect an upper bound
of the form \eqref{eq:slicing-quermassintegral} even if sections are replaced by projections. More precisely,
if $n\gr 2$, $1\ls k\ls n$ and $0\ls j\ls n-k-1$ then for every $\alpha >0$ there exists a convex body $K$ in ${\mathbb R}^n$
such that
\begin{equation}\label{eq:main-negative}W_j(K)>\alpha \,|K|^{\frac{k}{n}}\,\max_{F\in G_{n,n-k}}W_j(P_F(K))
\gr \alpha \,|K|^{\frac{k}{n}}\,\max_{F\in G_{n,n-k}}W_j(K\cap F),\end{equation}
where $P_F(K)$ denotes the orthogonal projection of $K$ onto $F$.

In Section~\ref{section-3} we explore the possibility to obtain some positive results for variants of this question. Our first main
result is the next theorem for the class of centered convex bodies.

\begin{theorem}\label{th:intro-1}Let $K$ be a centered convex body in ${\mathbb R}^n$. For every $0\ls j\ls n-k-1\ls n-1$
we have that
$$\alpha_{n,k,j}\left(\frac{n-k-j+1}{n+1}\right)^{n-k-j}W_j(K)
\ls \int_{G_{n,n-k}}|P_{F^{\perp }}(K)|W_j(K\cap F)d\nu_{n,n-k}(F)\ls \alpha_{n,k,j}\binom{n-j}{k}W_j(K),$$
where $\alpha_{n,k,j}=\frac{\omega_{n-k}\omega_{n-j}}{\omega_{n-k-j}\omega_n}$.
\end{theorem}

Using the left hand side inequality of Theorem~\ref{th:intro-1} one immediately gets
$$\alpha_{n,k,j}\left(\frac{n-k-j+1}{n+1}\right)^{n-k-j}W_j(K)
\ls s_k(K)\cdot\max_{F\in G_{n,n-k}}W_j(K\cap F),$$
where
$$s_k(K):=\int_{G_{n,k}}|P_{F}(K)|d\nu_{n,k}(F).$$
By Kubota's formula (see \eqref{eq:kubota} in Section~\ref{section-2}) we have $s_k(K)=\frac{\omega_{k}}{\omega_n}W_{n-k}(K)$, and hence we obtain:

\begin{corollary}\label{cor:intro}Let $K$ be a centered convex body in ${\mathbb R}^n$. For every $0\ls j\ls n-k-1\ls n-1$
we have that
\begin{equation}\label{eq:positive-1}\gamma_{n,k,j}\left(\frac{n-k-j+1}{n+1}\right)^{n-k-j}W_j(K)
\ls W_{n-k}(K)\cdot\max_{F\in G_{n,n-k}}W_j(K\cap F),\end{equation}
where $\gamma_{n,k,j}=\frac{\omega_{n-k}\omega_{n-j}}{\omega_{n-k-j}\omega_k}$.
\end{corollary}

This estimate, which is in the spirit of results from \cite{Brazitikos-Liakopoulos-TAMS}, shows that we can have some
version of \eqref{eq:slicing-quermassintegral} if we make the additional assumption that the natural parameter
$W_{n-k}^{1/k}(K)/|K|^{1/n}$ of the body is well-bounded.

A different variant of \eqref{eq:positive-1} is obtained in \cite{Brazitikos-Liakopoulos-TAMS}, which involves the parameter
$$t(K)=\left(\frac{|K|}{|r(K)B^n_2|}\right)^{\frac{1}{n}}$$
where $r(K)$ is the inradius of $K$, i.e. the largest value of $r>0$ for which there exists $x_0\in K$ such that $x_0+rB_2^n\subseteq K$.
It is proved in \cite{Brazitikos-Liakopoulos-TAMS} that if $K$ is a convex body in $\mathbb{R}^n$ with $0\in {\rm int}(K)$
then, for all $1\ls j\ls n-k\ls n-1$ we have that
\begin{equation}\label{eq:positive-2}\frac{W_j(K)}{|K|}\ls t(K)^j\max_{F\in G_{n,n-k}}\frac{W_j(K\cap F)}{|K\cap F|}.\end{equation}
Here, we improve \eqref{eq:positive-2} in two ways. First, we compare $\frac{W_j(K)}{|K|}$ with the average of
$\frac{W_j(K\cap F)}{|K\cap F|}$ over all $F\in G_{n,n-k}$. Second, we show that the two quantities are equivalent up to constants
depending only on the dimensions $n,k$ and $j$, with no dependence on any other parameter of the body.

\begin{theorem}\label{th:intro-2}Let $K$ be a centrally symmetric convex body in ${\mathbb R}^n$. For any $0\ls j\ls n-k-1\ls n-1$
we have that
$$\alpha_{n,k,j}\binom{n}{k}^{-1}\frac{W_j(K)}{|K|}\ls\int_{G_{n,n-k}}\frac{W_j(K\cap F)}{|K\cap F|}d\nu_{n,n-k}(F)
\ls \alpha_{n,k,j}\binom{n-j}{k}\frac{W_j(K)}{|K|}.$$
\end{theorem}

In Section~\ref{section-4} we study a variant of the slicing problem. Our starting point is a surface area variant of the equivalence of the
isomorphic Busemann--Petty problem with the slicing problem:
Assuming that there is a constant $\gamma_n$ such that
if $K$ and $D$ are centrally symmetric convex bodies in $\mathbb{R}^n$ that satisfy
$$S(K\cap\xi^{\perp})\ls S(D\cap\xi^{\perp})$$ for all $\xi\in S^{n-1}$, then $S(K)\ls \gamma_n S(D)$,
one can see that there is some constant $c(n)$ such that
\begin{equation}\label{slicing-15}
S(K)\ls c(n)S(K)^{\frac{1}{n-1}}\max_{\xi\in S^{n-1}} S(K\cap\xi^{\perp})
\end{equation}
for every convex body $K$ in ${\mathbb R}^n$. It was proved in \cite{Brazitikos-Liakopoulos-TAMS}
that an inequality of this type holds true in general. If $K$ is a convex body in ${\mathbb R}^n$ then
$$S(K)\ls A_nS(K)^{\frac{1}{n-1}}\max\limits_{\xi\in S^{n-1}}S(K\cap\xi^{\perp })$$
where $A_n>0$ is a constant depending only on $n$. Actually, the result is first proved for an arbitrary ellipsoid
and then it is extended to any convex body, using John's theorem. Here, we give a direct proof of a more general
result showing that an inequality analogous to \eqref{slicing-15} holds for any $k$ and $j$, where $k$ is the codimension
of the subspaces and $j$ is the order of the quermassintegral that we consider.
Note that \eqref{slicing-15} corresponds to the case $k=j=1$ of the next theorem.

\begin{theorem}\label{th:intro-3}
Let $K$ be a centrally symmetric convex body in $\mathbb{R}^n$. For every $0\ls j\ls n-k-1\ls n-1$ we have that
\begin{equation*}
W_j(K)^{n-k-j}\ls \alpha_{n,k,j}\max_{F\in G_{n,n-k}}W_j(K\cap F)^{n-j},
\end{equation*}
where $\alpha_{n,k,j}>0$ is a constant depending only on $n,k$ and $j$.
\end{theorem}

Theorem~\ref{th:intro-3} is a consequence of another double-sided inequality. Using the generalized
Blaschke-Petkantschin formula and integral-geometric results of Dann, Paouris and Pivovarov from \cite{D-P-P}
we show that if $K$ is a centrally symmetric convex body in $\mathbb{R}^n$ then for all $0\ls j\ls n-k-1\ls n-1$ we have that
\begin{equation}\label{eq:BP}
c_{n,k,j}|K|^{n-k}W_{k+j}(K)\ls\int_{K}\cdots\int_{K}W^{(n-k)}_j({\rm conv}\{0,x_1,\ldots ,x_{n-k}\})dx_{n-k}\cdots dx_1\ls\delta_{n,k,j}|K|^{n-k}W_{k+j}(K),
\end{equation}
where $c_{n,k,j}$ and $\delta_{n,k,j}$ are constants depending only on $n,k$ and $j$.

Using a variant of \eqref{eq:BP} we also obtain the next result.

\begin{theorem}\label{th:intro-4}Let $K$ be a centrally symmetric convex body in ${\mathbb R}^n$.
Then, for all $0\ls j\ls n-1$ and $N\gr n+1$ we have that
\begin{equation*}
c_{n,N,j}|K|^NW_j(K)\ls\int_{K}\cdots\int_{K}W_j({\rm conv}\{x_1,\ldots ,x_N\})\,dx_N\cdots dx_1
\end{equation*}
where $c_{n,N,j}$ is a constant depending only on $n,N$ and $j$.
\end{theorem}

Theorem~\ref{th:intro-4} is related to a result of Hartzoulaki and Paouris from \cite{Hartzoulaki-Paouris-2003}
which asserts that, among all convex bodies $K$ of volume $1$ in ${\mathbb R}^n$, the expected value
$$\int_{K}\cdots\int_{K}W_j({\rm conv}\{x_1,\ldots ,x_N\})\,dx_N\cdots dx_1$$
is minimized when $K$ is the Euclidean ball $D_n$ of volume $1$. In other words, for every convex body $K$ in ${\mathbb R}^n$
we have
\begin{equation}\label{eq:hartzoulaki}A_{n,N,j}|K|^{N+\frac{n-j}{n}}\ls\int_{K}\cdots\int_{K}W_j({\rm conv}\{x_1,\ldots ,x_N\})\,dx_N\cdots dx_1\end{equation}
for some constant $A_{n,N,j}>0$, with equality if and onlly if $K=D_n$ (see also \cite{Saroglou-2010} for the characterization
of the cases of equality). This inequality generalizes to the setting of quermassintegrals
a well-known result of Groemer \cite{Groemer-1974} (see also \cite{Schopf-1977} and \cite{Giannopoulos-Tsolomitis-2003})
which concerned the expected value of the volume of ${\rm conv}\{x_1,\ldots ,x_N\}$. All these results are proved via
Steiner symmetrization, and hence, the left hand-side of \eqref{eq:hartzoulaki} involves only the volume of $K$ and
is not sensitive to the value of $W_j(K)$. Since
$$\omega_n^\frac{j}{n}|K|^\frac{n-j}{n}\ls W_j(K)$$
by the Aleksandrov inequalities, Theorem~\ref{th:intro-4} provides a stronger estimate if we ignore the values
of the constants.

\section{Mixed volumes and Quermassintegrals}\label{section-2}

We work in ${\mathbb R}^n$, which is equipped with the standard inner product $\langle\cdot ,\cdot\rangle $. We denote by $\|\cdot \|_2$
the Euclidean norm, and write $B_2^n$ for the Euclidean unit ball and $S^{n-1}$ for the unit sphere.
Volume is denoted by $|\cdot |$. We write $\omega_n$ for the volume of $B_2^n$ and $\sigma $ for the rotationally invariant probability
measure on $S^{n-1}$. The Grassmann manifold $G_{n,k}$ of all $k$-dimensional subspaces of ${\mathbb R}^n$ is equipped with the Haar probability
measure $\nu_{n,k}$. For every $1\ls k\ls n-1$ and $F\in G_{n,k}$ we write $P_F$ for the orthogonal projection from $\mathbb R^{n}$ onto $F$.

\smallskip

A convex body in ${\mathbb R}^n$ is a compact convex subset $K$ of ${\mathbb R}^n$ with non-empty interior. We say that $K$ is
centrally symmetric if $x\in K$ implies that $-x\in K$, and that $K$ is centered if its barycenter $\frac{1}{|K|}\int_Kx\,dx $ is at the origin.
The support function of a convex body $K$ is defined by $h_K(y)=\max \{\langle x,y\rangle :x\in K\}$, and the mean width of $K$ is
\begin{equation*}w(K)=\int_{S^{n-1}}h_K(\xi )\,d\sigma (\xi ). \end{equation*}
The circumradius of $K$ is the quantity $R(K)=\max\{ \|x\|_2:x\in K\}$ i.e. the smallest $R>0$ for which $K\subseteq RB_2^n$.
We write $r(K)$ for the inradius of $K$, the largest $r>0$
for which there exists $x_0\in K$ such that $x_0+rB_2^n\subseteq K$. If $0\in {\rm int}(K)$ then we define the polar body $K^{\circ }$ of $K$ by
\begin{equation*}K^{\circ }:=\{ y\in {\mathbb R}^n: \langle x,y\rangle \ls 1 \;\hbox{for all}\; x\in K\}. \end{equation*}
The cone probability measure $\mu_K$ on the boundary $\partial (K)$ of a convex body $K$ with $0\in {\rm int}(K)$ is defined by
$$\mu_K(B)=\frac{|\{rx:x\in B,0\ls r\ls 1\}|}{|K|}$$
for all Borel subsets $B$ of $\partial (K)$. We shall use the identity
\begin{equation}\label{eq:cone-measure}\int_{{\mathbb R}^n}f(x)\,dx=n\,|K|\int_0^{\infty }r^{n-1}\int_{\partial (K)}f(rx)\,d\mu_K(x)\,dr\end{equation}
which holds for every integrable function $f:{\mathbb R}^n\to {\mathbb R}$ (see \cite[Proposition~1]{Naor-Romik}).

\smallskip

Minkowski's theorem, and definition of the mixed volumes, states that if $K_1,\ldots ,K_m$ are non-empty, compact convex
subsets of ${\mathbb R}^n$, then the volume of $\lambda_1K_1+\cdots +\lambda_mK_m$ is a homogeneous polynomial of degree $n$ in
$\lambda_i>0$. One can write
\begin{equation*}|\lambda_1K_1+\cdots +\lambda_mK_m|=\sum_{1\ls i_1,\ldots ,i_n\ls m}
V(K_{i_1},\ldots ,K_{i_n})\lambda_{i_1}\cdots \lambda_{i_n},\end{equation*}
where the coefficients $V(K_{i_1},\ldots ,K_{i_n})$ are invariant under permutations of their arguments. The coefficient $V(K_{i_1},\ldots ,K_{i_n})$
is the mixed volume of $K_{i_1},\ldots ,K_{i_n}$. In particular, if $K$ and $D$ are two convex bodies in ${\mathbb R}^n$
then the function $|K+\lambda D|$ is a polynomial in $\lambda\in [0,\infty )$:
\begin{equation*}|K+\lambda D|=\sum_{j=0}^n \binom{n}{j} V_{n-j}(K,D)\;\lambda^j,\end{equation*}
where $V_{n-j}(K,D)= V((K,n-j),(D,j))$ is the $j$-th mixed volume of $K$ and $D$ (we use  the notation $(D,j)$ for $D,\ldots ,D$ $j$-times).
If $D=B_2^n$ then we set $W_j(K):=V_{n-j}(K,B_2^n)=V((K, n-j), (B_2^n, j))$; this is the $j$-th quermassintegral of $K$.
The intrinsic volumes $V_j(K)$ of $K$ are defined for $0\ls j\ls n$ by
\begin{equation}\label{eq:intrinsic}V_j(K)=\frac{\binom{n}{j}}{\omega_{n-j}}W_{n-j}(K).\end{equation}
Thus, Steiner's formula can be written in the following two equivalent ways:
$$|K+\lambda B_2^n|=\sum_{j=0}^n\binom{n}{j}W_j(K)\lambda^j=\sum_{j=0}^n\omega_{n-j}V_j(K)\lambda^{n-j}.$$
The mixed volume $V_{n-1}(K,D)$ can be expressed as
\begin{equation}\label{eq:not-1}V_{n-1}(K,D)={\frac{1}{n}}\int_{S^{n-1}}h_D(\theta )d\sigma_K(\theta ),\end{equation}
where $\sigma_K$ is the surface area measure of $K$; this is the Borel measure
on $S^{n-1}$ defined by
\begin{equation*}\sigma_K(A)=m (\{x\in {\rm bd}(K):\;{\rm the}\;{\rm outer}
\;{\rm normal}\;{\rm to}\;K\;{\rm at}\;x\;{\rm belongs}\;
{\rm to}\;A\}),\end{equation*}
where $m$ is the Hausdorff measure on ${\rm bd}(K)$. In particular, the surface area $S(K):=\sigma_K(S^{n-1})$ of $K$ satisfies
\begin{equation*}S(K)=nW_1(K).\end{equation*}
Kubota's integral formula expresses the quermassintegral $W_j(K)$ as an average of the volumes of
$(n-j)$-dimensional projections of $K$:
\begin{equation}\label{eq:kubota}W_j(K)=\frac{\omega_n}{\omega_{n-j}}\int_{G_{n,n-j}}
|P_F(K)|d\nu_{n,n-j}(F).\end{equation}Applying this formula for
$j=n-1$ we see that \begin{equation*}W_{n-1}(K)=\omega_n
w(K).\end{equation*}
Aleksandrov's inequalities (see \cite{Burago-Zalgaller-book} and \cite{Schneider-book}) imply that if we set
\begin{equation}\label{eq:aleksandrov-1}Q_j(K)=\left
(\frac{W_{n-j}(K)}{\omega_n}\right )^{\frac{1}{j}}=\left
(\frac{1}{\omega_j}\int_{G_{n,j}}|P_F(K)|\,d\nu_{n,j}(F)\right )^{\frac{1}{j}},\end{equation}
then $j\mapsto Q_j(K)$ is decreasing. In particular, for every $1\ls j\ls n-1$ we have
\begin{equation}\label{eq:aleksandrov-2}{\rm vrad}(K):=\left (\frac{|K|}{\omega_n}\right )^{\frac{1}{n}}\ls \left (\frac{1}{\omega_j}\int_{G_{n,j}}|P_F(K)|\,d\nu_{n,j}(F)\right )^{\frac{1}{j}}\ls w(K).\end{equation}
Another consequence of Aleksandrov's inequalities is the following Brunn-Minkowski inequality
for the quermassintegrals: one has
\begin{equation}\label{eq:Wj-bm}W_j(K+D)^{\frac{1}{n-j}}\gr W_j(K)^{\frac{1}{n-j}}+W_j(D)^{\frac{1}{n-j}}\end{equation}
for all $j=0,\ldots ,n-1$ and any pair of convex bodies $K$ and $D$ in ${\mathbb R}^n$.

\smallskip

We refer to the books \cite{Gardner-book} and \cite{Schneider-book} for basic facts from the Brunn-Minkowski theory and to the books
\cite{AGA-book} and \cite{BGVV-book} for basic facts from asymptotic convex geometry.

\section{Proof of Theorem~\ref{th:intro-1} and Theorem~\ref{th:intro-2}}\label{section-3}

In this section we prove Theorem~\ref{th:intro-1} and Theorem~\ref{th:intro-2}.
Consider the set $A_{n,k}$ of $k$-dimensional affine subspaces of ${\mathbb R}^n$
equipped with the Haar probability measure $\mu_{n,k}$. Our starting point is Crofton's formula
(see \cite[Theorem~5.1.1]{Schneider-Weil-book}).

\begin{theorem}\label{th:sw}Let $K$ be a convex body in ${\mathbb R}^n$. For any $1\ls j\ls k\ls n-1$ we have that
$$\int_{A_{n,k}}V_j(K\cap F)\,d\mu_{n,k}(F)=\frac{\binom{k}{j}\omega_k\omega_{n-k+j}}{\binom{n}{k-j}\omega_j\omega_n}V_{n-k+j}(K).$$
\end{theorem}

Using \eqref{eq:intrinsic} for $K$ and $K\cap F$, $F\in G_{n,k}$, we rewrite the assertion of Theorem~\ref{th:sw} as follows:
$$\int_{A_{n,k}}W_{k-j}(K\cap F)\,d\mu_{n,k}(F)=\frac{\omega_k\omega_{n-k+j}}{\omega_j\omega_n}W_{k-j}(K).$$
Equivalently, for all $0\ls j\ls n-k-1\ls n-1$,
$$\int_{A_{n,n-k}}W_j(K\cap F)\,d\mu_{n,n-k}(F)=\frac{\omega_{n-k}\omega_{n-j}}{\omega_{n-k-j}\omega_n}W_j(K).$$
Note that (see e.g. \cite[Section~5.1]{Schneider-Weil-book})
\begin{align}\label{eq:04}\int_{A_{n,n-k}}W_j(K\cap F)\,d\mu_{n,n-k}(F)&=\int_{G_{n,n-k}}\int_{F^{\perp }}W_j(K\cap (x+F))\,dx\,d\nu_{n,n-k}(F)\\
\nonumber &=\int_{G_{n,k}}\int_{F}W_j(K\cap (x+F^{\perp }))\,dx\,d\nu_{n,k}(F).
\end{align}
We shall use the next set of inequalities of Rogers--Shephard type.

\begin{lemma}\label{lem:rs}Let $K$ be a centered convex body in ${\mathbb R}^n$. For every
$F\in G_{n,k}$ and $0\ls j\ls n-k-1\ls n-1$ we have that
\begin{align*}
\binom{n-j}{k}^{-1}|P_F(K)|\,W_j(K\cap F^{\perp }) &\ls \int_FW_j(K\cap (x+F^{\perp }))\,dx \\
&\ls \left(\frac{n+1}{n-k-j+1}\right)^{n-k-j}|P_F(K)|W_j(K\cap F^{\perp }).
\end{align*}
\end{lemma}

\begin{proof}We follow the proof of the Rogers-Shephard inequality (see \cite{Rogers-Shephard-1958}).
For the left hand-side inequality, let $x\in P_F(K)$ and $t=\|x\|_{P_F(K)}\in [0,1]$,
where $\|\cdot\|_{P_F(K)}$ is the Minkowski functional of $P_F(K)$. Then, there exists $z\in F^{\perp}$ such that $x+tz\in tK$, therefore
$$K\cap (x+F^{\perp })\supseteq (1-t)(K\cap F^{\perp })+x+tz.$$
Then,
\begin{align*}
W_j(K\cap (x+F^{\perp })) &= V(K\cap (x+F^{\perp });n-k-j,B_2^{n-k};j)\gr V((1-t)(K\cap F^{\perp });n-k-j,B_2^{n-k};j) \\
&= (1-t)^{n-k-j}V(K\cap F^{\perp };n-k-j,B_2^{n-k};j)=(1-t)^{n-k-j}W_j(K\cap F^{\perp }).
\end{align*}
It follows that
$$W_j(K\cap (x+F^{\perp }))\gr (1-\|x\|_{P_F(K)})^{n-k-j}W_j(K\cap F^{\perp }).$$
Integrating with respect to $x\in F$ we obtain
\begin{align*}\int_FW_j(K\cap (x+F^{\perp }))\,dx &= \int_{P_F(K)}W_j(K\cap (x+F^{\perp }))\,dx\\
&\gr W_j(K\cap F^{\perp })\int_F(1-\|x\|_{P_F(K)})^{n-k-j}\mathds{1}_{P_F(K)}(x)\,dx,
\end{align*}
because the support of $x\mapsto W_j(K\cap (x+F^{\perp }))$ is the projection $P_F(K)$ of $K$ onto $F$.
Now, changing to polar coordinates with respect to the cone measure $\mu_{\partial P_F(K)}$ of $\partial P_F(K)$
and using \eqref{eq:cone-measure}, we write
\begin{align*}
\int_F(1-\|x\|_{P_F(K)})^{n-k-j}\mathds{1}_{P_F(K)}(x)\,dx &= k|P_F(K)|\int_{\partial P_F(K)}\int_0^1t^{k-1}(1-\|t\theta \|_{P_F(K)})^{n-k-1}dt\,d\mu_{\partial P_F(K)}(\theta )\\
&= k|P_F(K)|\int_0^1t^{k-1}(1-t)^{n-k-j}dt=\binom{n-j}{k}^{-1}|P_F(K)|.
\end{align*}
This proves the left hand-side inequality of the lemma. For the right hand-side inequality we write
$$\int_FW_j(K\cap (x+F^{\perp }))\,dx =\int_{P_F(K)}W_j(K\cap (x+F^{\perp }))\,dx
\ls |P_F(K)|\,\max_{x\in F}W_j(K\cap (x+F^{\perp })).$$
Finally we use an inequality of Stephen and Yaskin (see \cite{Stephen-Yaskin-2019}) for the quermassintegrals of centered convex bodies, namely,
$$\max_{x\in F}W_j(K\cap (x+F^{\perp }))\ls \left(\frac{n+1}{n-k-j+1}\right)^{n-k-j}W_j(K\cap F^{\perp }),$$
which extends Fradelizi's inequality
\begin{equation}\label{eq:fradelizi-1}\max_{x\in F}|K\cap (x+F^{\perp })|\ls \left(\frac{n+1}{n-k+1}\right)^{n-k}|K\cap F^{\perp }|\end{equation}
from \cite{Fradelizi-1997} that had settled the volume case.\end{proof}

From Lemma~\ref{lem:rs} we can easily deduce Theorem~\ref{th:intro-1}.

\begin{proof}[Proof of Theorem~$\ref{th:intro-1}$]Let $K$ be a centered convex body in ${\mathbb R}^n$. For every $0\ls j\ls n-k-1\ls n-1$
We integrate the inequalities of Lemma~\ref{lem:rs} with respect to $\nu_{n,k}$ and recalling \eqref{eq:04}
we obtain
\begin{equation}\label{eq:first-1}\alpha_{n,k,j}\left(\frac{n-k-j+1}{n+1}\right)^{n-k-j}W_j(K)
\ls \int_{G_{n,k}}|P_F(K)|W_j(K\cap F^{\perp})d\nu_{n,k}(F)\ls \alpha_{n,k,j}\binom{n-j}{k}W_j(K),\end{equation}
where $\alpha_{n,k,j}=\frac{\omega_{n-k}\omega_{n-j}}{\omega_{n-k-j}\omega_n}$. Since
$$\int_{G_{n,k}}|P_F(K)|W_j(K\cap F^{\perp})d\nu_{n,k}(F)=\int_{G_{n,n-k}}|P_{F^{\perp}}(K)|W_j(K\cap F)d\nu_{n,n-k}(F),$$
this is equivalent to the assertion of the theorem.\end{proof}

\smallskip

We pass now to the proof of Theorem~\ref{th:intro-2}. First we prove a version of the result for centered convex bodies.

\begin{theorem}\label{th:next}Let $K$ be a centered convex body in ${\mathbb R}^n$. For any $0\ls j\ls n-k-1\ls n-1$
we have
$$\beta_{n,k,j}\frac{W_j(K)}{|K|}\ls\int_{G_{n,k}}\frac{W_j(K\cap F^{\perp })}{|K\cap F^{\perp }|}d\nu_{n,k}(F)
\ls \gamma_{n,k,j}\frac{W_j(K)}{|K|},$$
where $\beta_{n,k,j}=\alpha_{n,k,j}\left(\frac{n-k-j+1}{n+1}\right)^{n-k-j}\binom{n}{k}^{-1}$ and
$\gamma_{n,k,j}=\alpha_{n,k,j}\binom{n-j}{k}\left(\frac{n+1}{n-k+1}\right)^{n-k}$.
\end{theorem}

\begin{proof}From the left hand-side inequality of Theorem~\ref{th:intro-1} and the Rogers-Shephard type inequality
\begin{equation}\label{eq:spingarn}|P_F(K)|\,|K\cap F^{\perp}|\ls\binom{n}{k}|K|\end{equation}
of Spingarn (see \cite{Spingarn-1993}) we obtain the left hand-side inequality of the theorem.
Next, using Fubini's theorem and \eqref{eq:fradelizi-1} we see that
\begin{align}\label{eq:fradelizi-2}|K| &=\int_{P_F(K)}|K\cap (x+F^{\perp})|\,dx \ls |P_F(K)|\,\max_{x\in P_F(K)}|K\cap (x+F^{\perp})|\\
 &\ls \left(\frac{n+1}{n-k+1}\right)^{n-k}|P_F(K)|\,|K\cap F^{\perp}|,\nonumber\end{align}
and combining this estimate with the right hand-side inequality of Theorem~\ref{th:intro-1}
we obtain the right hand side inequality of the theorem.\end{proof}

For the centrally symmetric case we may use some additional observations.

\begin{proof}[Proof of Theorem~$\ref{th:intro-2}$] Assuming that $K$ is centrally symmetric, we have
$$W_j(K\cap (x+F^{\perp }))=W_j(K\cap (-x+F^{\perp }))$$
for every $x\in P_F(K)$. We also have $\frac{1}{2}[K\cap (x+F^{\perp })]+\frac{1}{2}[K\cap (-x+F^{\perp })]\subseteq K\cap F^{\perp }$,
and hence, by the Brunn-Minkowski inequality \eqref{eq:Wj-bm} we get
$$W_j(K\cap (x+F^{\perp }))\ls W_j(K\cap F^{\perp })$$
for every $x\in P_F(K)$. It follows that
$$\int_FW_j(K\cap (x+F^{\perp }))\,dx\ls |P_F(K)|W_j(K\cap F^{\perp }).$$
Integrating over all $F\in G_{n,k}$ and using \eqref{eq:04} we get
$$\alpha_{n,k,j}W_j(K)\ls \int_{G_{n,k}}|P_F(K)|W_j(K\cap F^{\perp })d\nu_{n,k}(F).$$
Applying \eqref{eq:spingarn} we see that
$$\alpha_{n,k,j}\binom{n}{k}^{-1}\frac{W_j(K)}{|K|}\ls\int_{G_{n,k}}\frac{W_j(K\cap F^{\perp })}{|K\cap F^{\perp }|}d\nu_{n,k}(F).$$
Then, using the right hand side inequality of Theorem~\ref{th:intro-1} and the fact that
$$|K|\ls |P_F(K)|\,\max_{x\in P_F(K)}|K\cap (x+F^{\perp})|=|P_F(K)|\,|K\cap F^{\perp }|$$
holds for centrally symmetric convex bodies, we conclude that
$$\int_{G_{n,k}}\frac{W_j(K\cap F^{\perp })}{|K\cap F^{\perp }|}d\nu_{n,k}(F)
\ls \alpha_{n,k,j}\binom{n-j}{k}\frac{W_j(K)}{|K|}.$$
This shows that for any $0\ls j\ls n-k-1\ls n-1$ we have
$$\alpha_{n,k,j}\binom{n}{k}^{-1}\frac{W_j(K)}{|K|}\ls\int_{G_{n,k}}\frac{W_j(K\cap F^{\perp })}{|K\cap F^{\perp }|}d\nu_{n,k}(F)
\ls \alpha_{n,k,j}\binom{n-j}{k}\frac{W_j(K)}{|K|},$$
which is equivalent to the assertion of the theorem.\end{proof}

\section{Proof of Theorem~\ref{th:intro-3}}\label{section-4}

In this section we prove Theorem~\ref{th:intro-3}. We start by introducing a number of tools
that will be needed. Let $f_1,\ldots ,f_q$ be non-negative, bounded, integrable functions on $\mathbb{R}^d$ such that $\|f_j\|_1>0$ for every
$j=1,\ldots ,q$. Given a compact convex set $C\subset\mathbb{R}^q$ and $p\neq0$, we define
\begin{equation*}
{\mathcal F}_{C,p}(f_1,\ldots ,f_q)=\Bigg(\int_{\mathbb{R}^d}\cdots\int_{\mathbb{R}^d}|[x_1,\ldots ,x_q]C|^p\prod_{j=1}^{q}\frac{f_j(x_j)}{\|f_j\|_1}dx_1\cdots dx_q\Bigg)^{1/p},
\end{equation*}
where
\begin{equation*}
[x_1,\ldots ,x_q]C=\bigg\{\sum_{j=1}^{q}c_jx_j:c=(c_j)\in C\bigg\}
\end{equation*}
and $|[x_1,\ldots ,x_q]C|$ is the volume of this set. The next theorem is due to Dann, Paouris and Pivovarov (see \cite[Theorem~4.4]{D-P-P}).

\begin{theorem}\label{th:new-1}
Let $q$ and $d$ be positive integers. Let $f$ be a non-negative, bounded integrable function on $\mathbb{R}^d$ with $\|f\|_1>0$.
Let $C\subset\mathbb{R}^q$ be a compact convex set and $p\gr1$. We set $m=\min(q,d,{\rm dim}(C))$. Then,
\begin{equation*}
\mathcal F_{C,p}(f;q)\gr\Big(\frac{\|f\|_1}{\omega_d\|f\|_\infty}\Big)^\frac{m}{d}\mathcal F_{C,p}(\mathds{1}_{B^d_2};q),
\end{equation*}
where $\mathcal F_{C,p}(f;q)=\mathcal F_{C,p}(f,\ldots ,f)$ {\rm (}$q$ times{\rm )}.
\end{theorem}

We also introduce some notation. Given a compact convex set $L\subset\mathbb{R}^m\subset\mathbb{R}^n$, we denote by $W^{(n)}_j(L)$
the $j$-th quermassintegral of $L$ in dimension $n$, which is defined by
\begin{equation*}
W^{(n)}_j(L)=V(L;n-j,B^n_2;j).
\end{equation*}
In the case where the $j$-th quermassintegral of $L$ is taken in dimension $m$ our notation will be the usual one. Namely,
\begin{equation*}
W_j(L)=W^{(m)}_j(L)=V(L;m-j,B^m_2;j).
\end{equation*}
The next lemma gives the relation between $W^{(n)}_j(L)$ and $W^{(m)}_j(L)$. It is probably well-known but we include
a proof for completeness.

\begin{lemma}\label{lem:new-2}
Let $K$ be a compact convex set in $\mathbb{R}^n$ and $F\in G_{n,n-k}$. Then,
\begin{equation*}
W^{(n)}_{k+j}(K\cap F)=\frac{\omega_{k+j}\binom{n-k}{j}}{\omega_j\binom{n}{k+j}}W^{(n-k)}_j(K\cap F)
\end{equation*}
for all $j=1,\ldots ,n-k$.
\end{lemma}

\begin{proof}We can give a proof by induction on $k$. Note that if we consider
a compact convex set $T\subset {\mathbb R}^{n-1}\subset {\mathbb R}^n$ then
 \begin{align*}
\sum_{j=0}^n \binom{n}{j} W_j^{(n)}(T) \lambda^j  &=|T+\lambda B_2^n|=\int_{-\lambda}^{\lambda} |(T+ \lambda B_{2}^{n}) \cap ({\mathbb R}^{n-1} + te_n)|_{n-1} dt\\
&=\int_{-\lambda}^{\lambda} |T + (\lambda^2 - t^2)^{1/2} B_{2}^{n-1}|_{n-1}dt\\
&=\sum_{j=0}^{n-1}\binom{n-1}{j} W_j^{(n-1)}(T)\int_{-\lambda}^{\lambda} (\lambda^2 - t^2)^{j/2}dt\\
&=\sum_{j=0}^{n-1}\binom{n-1}{j} W_j^{(n-1)}(T)\lambda^{j+1} \frac{\omega_{j+1}}{\omega_{j}}.
\end{align*}
Comparing the coefficients of the two polynomials we see that
$$W_{j+1}^{(n)}(T) = \frac{\omega_{j+1}}{\omega_j}\frac{\binom{n-1}{j}}{\binom{n}{j+1}}W_j^{(n-1)}(T)$$
for all $j=1,\ldots ,n-1$. Now, induction shows that if $T\subset {\mathbb R}^{n-1}\subset {\mathbb R}^n$ then
$$W^{(n)}_{k+j}(T)=\frac{\omega_{k+j}\binom{n-k}{j}}{\omega_j\binom{n}{k+j}}W^{(n-k)}_j(T)$$
for all $j=1,\ldots ,n-k$. The lemma follows if we identify $F\in G_{n,n-k}$ with ${\mathbb R}^{n-k}$
and apply the above formula to $T=K\cap F$.\end{proof}

With these tools we are able to prove the next theorem.

\begin{theorem}\label{th:new-3}
Let $K$ be a centrally symmetric convex body in $\mathbb{R}^n$. For all $0\ls j\ls n-k-1\ls n-1$ we have that
\begin{equation*}
c_{n,k,j}|K|^{n-k}W_{k+j}(K)\ls\int_{K}\cdots\int_{K}W^{(n-k)}_j({\rm conv}\{0,x_1,\ldots ,x_{n-k}\})dx_{n-k}\cdots dx_1\ls\delta_{n,k,j}|K|^{n-k}W_{k+j}(K),
\end{equation*}
where $c_{n,k,j}=\frac{\omega_j}{\omega_{k+j}\omega_{n-k-j}}\frac{1}{\binom{n-k}{j}}\mathcal F_{C,1}(\mathds{1}_{B^{n-k-j}_2};n-k)$
with $C={\rm conv}\{0,e_1,\ldots ,e_{n-k}\}$ and
$$\delta_{n,k,j}=\frac{\omega_j}{\omega_{k+j}}\frac{\binom{n}{k+j}}{\binom{n-k}{j}}.$$
\end{theorem}

\begin{proof}
First we prove the left hand side inequality. Let $F\in G_{n,n-k-j}$ and define $f:F\longrightarrow\mathbb{R}^+$ by $$f(x)=\pi_F(\mathds{1}_K)(x)=|K\cap(x+F^\perp)|.$$
Then $\|f\|_1=\int_{F}|f(x)|dx=\int_{F}|K\cap(x+F^\perp)|dx=|K|$ and $\|f\|_{\infty}=|K\cap F^\perp|$.
We set $d=n-k-j$, $q=n-k$, $p=1$, $C={\rm conv}\{0,e_1,\ldots ,e_{n-k}\}$ and $m=\min\{n-k,n-k-j, n-k\}=n-k-j$. From
Theorem~\ref{th:new-1} we have that
\begin{align*}
&\frac{1}{|K|^{n-k}}\int_{K}\cdots\int_{K}|P_F({\rm conv}\{0,x_1,\ldots ,x_{n-k}\})|dx_{n-k}\cdots dx_1\\
&\hspace*{1.5cm}=\int_{\mathbb{R}^n}\cdots\int_{\mathbb{R}^n}|{\rm conv}\{0,P_F(x_1),\ldots ,P_F(x_{n-k})\}|\prod_{j=1}^{n-k}\frac{\mathds{1}_K(x_j)}{|K|}dx_{n-k}\cdots dx_1\\
&\hspace*{1.5cm}=\int_{F}\cdots\int_{F}|{\rm conv}\{0,x_1,\ldots ,x_{n-k}\}|\prod_{j=1}^{n-k}\frac{\pi_F(\mathds{1}_K)(x_j)}{|K|}dx_{n-k}\cdots dx_1\\
&\hspace*{1.5cm}=\int_{F}\cdots\int_{F}|{\rm conv}\{0,x_1,\ldots ,x_{n-k}\}|\prod_{j=1}^{n-k}\frac{f(x_j)}{\|f\|_1}dx_{n-k}\cdots dx_1\\
&\hspace*{1.5cm}=\mathcal F_{C,1}(f;n-k)\gr\Big(\frac{\|f\|_1}{\omega_{n-k-j}\|f\|_\infty}\Big)^\frac{n-k-j}{n-k-j}\mathcal F_{C,1}(\mathds{1}_{B^{n-k-j}_2};n-k)\\
&\hspace*{1.5cm}=\frac{F_{C,1}(\mathds{1}_{B^{n-k-j}_2};n-k)}{\omega_{n-k-j}}\frac{|K|}{|K\cap F^\perp|}\\
&\hspace*{1.5cm}\gr\frac{F_{C,1}(\mathds{1}_{B^{n-k-j}_2};n-k)}{\omega_{n-k-j}\binom{n}{n-k-j}}|P_F(K)|,
\end{align*}
where the last inequality comes from the Rogers-Shephard inequality $|K|\ls |P_F(K)|\,|K\cap F^{\perp}|$. Integrating both sides of this
inequality over $G_{n,n-k-j}$ we get
\begin{equation*}
\frac{F_{C,1}(\mathds{1}_{B^{n-k-j}_2};n-k)}{\omega_{n-k-j}}|K|^{n-k}W_{k+j}(K)\ls\int_{K}\cdots\int_{K}W^{(n)}_{k+j}({\rm conv}\{0,x_1,\ldots ,x_{n-k}\})dx_{n-k}\cdots dx_1.
\end{equation*}
Using the formula of Lemma~\ref{lem:new-2} we obtain the result.

For the right hand side inequality we observe that $P_F({\rm conv}\{0,x_1,\dots,x_{n-k}\})\subseteq P_F(K)$ for every $x_1,\dots,x_{n-k}\in K$, therefore we may write
\begin{align*}
&\int_{K}\cdots\int_{K}|P_F({\rm conv}\{0,x_1,\ldots ,x_{n-k}\})|dx_{n-k}\cdots dx_1\ls \int_{K}\cdots\int_{K}|P_F(K)|dx_{n-k}\cdots dx_1\\
&\hspace*{1.5cm}=|K|^{n-k}|P_F(K)|.
\end{align*}
Integrating both sides of this inequality over $G_{n,n-k-j}$ as before, we get
\begin{equation*}
\int_{K}\cdots\int_{K}W^{(n)}_{k+j}({\rm conv}\{0,x_1,\ldots ,x_{n-k}\})dx_{n-k}\cdots dx_1\ls |K|^{n-k}W_{k+j}(K),
\end{equation*}
 and using the formula of Lemma~\ref{lem:new-2} we obtain the result.
\end{proof}

\smallskip

We shall combine Theorem~\ref{th:new-3} with the generalized Blaschke-Petkantschin formula (see \cite[Theorem~7.2]{Schneider-Weil-book}).

\begin{lemma}\label{lem:new-4}
Let $1\ls s\ls n-1$. There exists a constant $p(n,s)>0$ such that, for any non negative bounded Borel
measurable function $g:(\mathbb{R}^n)^s\rightarrow\mathbb{R}$,
\begin{align*}
&\int_{\mathbb{R}^n}\cdots\int_{\mathbb{R}^n}g(x_1,\ldots ,x_s)dx_s\cdots dx_1\\
&\hspace*{1cm}=p(n,s)\int_{G_{n,s}}\int_{F}\cdots\int_{F}g(x_1,\ldots ,x_s)|{\rm conv}\{0,x_1,\ldots ,x_s\}|^{n-s}dx_s\cdots dx_1\,d\nu_{n,s}(F).
\end{align*}
\end{lemma}

\begin{theorem}\label{th:new-5}
Let $K$ be a centrally symmetric convex body in $\mathbb{R}^n$. For every $0\ls j\ls n-k-1\ls n-1$ we have that
\begin{equation*}
W_{k+j}(K)\ls c_{n,k,j}^{-1}\max_{F\in G_{n,n-k}}W_j(K\cap F)
\end{equation*}
where $c_{n,k,j}$ is the constant of Theorem~$\ref{th:new-3}$.
\end{theorem}

\begin{proof}
From the left hand side inequality of Theorem~\ref{th:new-3} and the Blaschke-Petkantschin formula with $s=n-k$ we get
\begin{align*}
&c_{n,k,j}|K|^{n-k}W_{k+j}(K)\ls\int_{K}\cdots\int_{K}W^{(n-k)}_j({\rm conv}\{0,x_1,\ldots ,x_{n-k}\})dx_{n-k}\cdots dx_1\\
&\hspace*{1cm}=\int_{\mathbb{R}^n}\cdots\int_{\mathbb{R}^n}W^{(n-k)}_j({\rm conv}\{0,x_1,\ldots ,x_{n-k}\})\prod_{i=1}^{n-k}\mathds{1}_K(x_i)dx_{n-k}\cdots dx_1\\
&\hspace*{1cm}=p(n,n-k)\int_{G_{n,n-k}}\int_{F}\cdots\int_{F}W^{(n-k)}_j({\rm conv}\{0,x_1,\ldots ,x_{n-k}\})|{\rm conv}\{0,x_1,\ldots ,x_{n-k}\}|^k\prod_{i=1}^{n-k}\mathds{1}_K(x_i)\\
&\hspace*{3cm}dx_{n-k}\cdots dx_1\nu_{n,n-k}(F)\\
&\hspace*{1cm}=p(n,n-k)\int_{G_{n,n-k}}\int_{K\cap F}\cdots\int_{K\cap F}W^{(n-k)}_j({\rm conv}\{0,x_1,\ldots ,x_{n-k}\})|{\rm conv}\{0,x_1,\ldots ,x_{n-k}\}|^k\\
&\hspace*{3cm}dx_{n-k}\cdots dx_1\nu_{n,n-k}(F)\\
&\hspace*{1cm}\ls p(n,n-k)\int_{G_{n,n-k}}W_j(K\cap F)\int_{K\cap F}\cdots\int_{K\cap F}|{\rm conv}\{0,x_1,\ldots ,x_{n-k}\}|^kdx_{n-k}\cdots dx_1\nu_{n,n-k}(F)\\
&\hspace*{1cm}\ls \max_{F \in G_{n,n-k}}W_j(K\cap F)\cdot p(n,n-k)\int_{G_{n,n-k}}\int_{K\cap F}\cdots\int_{K\cap F}|{\rm conv}\{0,x_1,\ldots ,x_{n-k}\}|^k\\
&\hspace*{3cm}dx_{n-k}\cdots dx_1\nu_{n,n-k}(F)\\
&\hspace*{1cm}= |K|^{n-k}\max_{F\in G_{n,n-k}}W_j(K\cap F),
\end{align*}
which gives the claim of the theorem.
\end{proof}

Now, Theorem~\ref{th:intro-3} follows from Theorem~\ref{th:new-5} and Aleksandrov's inequalities.

\begin{proof}[Proof of Theorem~$\ref{th:intro-3}$]
Let $K$ be a centrally symmetric convex body in $\mathbb{R}^n$ and $0\ls j\ls n-k-1\ls n-1$. By Aleksandrov's inequalities we have
\begin{equation}
\omega_n^k W_j(K)^{n-k-j}\ls W_{k+j}(K)^{n-j}.
\end{equation}
Combining this fact with the inequality of Theorem~\ref{th:new-5} we get
\begin{equation*}W_j(K)^{n-k-j}\ls(\omega_n^k c_{n,k,j}^{n-j})^{-1}\max_{F \in G_{n,n-k}}W_j(K\cap F)^{n-j}\end{equation*}
and the proof is complete.\end{proof}

Assuming that $K$ is centered, we may repeat the proof of Theorem~\ref{th:new-3}, the only difference being that
for the function $f(x)=\pi_F(\mathds{1}_K)(x)=|K\cap(x+F^\perp)|$ we have
$$\|f\|_\infty\ls\Big(\frac{n+1}{k+j+1}\Big)^{k+j}|K\cap F^\perp|$$
by Fradelizi's inequality \eqref{eq:fradelizi-2}. Then, we have the following analogues of Theorem~\ref{th:new-5} and Theorem~\ref{th:intro-3}.

\begin{corollary}\label{cor:new-2}
Let $K$ be a centered convex body in $\mathbb{R}^n$. Then, for all $0\ls j\ls n-k-1\ls n-1$ we have that
\begin{equation*}
c'_{n,k,j}|K|^{n-k}W_{k+j}(K)\ls\int_{K}\cdots\int_{K}W^{(n-k)}_j({\rm conv}\{0,x_1,\ldots ,x_{n-k}\})\,dx_{n-k}\cdots dx_1
\ls\delta_{n,k,j}|K|^{n-k}W_{k+j}(K),
\end{equation*}
where $c^{\prime}_{n,k,j}=c_{n,k,j}\left(\frac{n+1}{k+j+1}\right)^{-(k+j)}$ and
$\delta_{n,k,j}$ is the constant in Theorem~$\ref{th:new-3}$.
\end{corollary}

\begin{corollary}\label{cor:new-3}
Let $K$ be a centered convex body in $\mathbb{R}^n$. Then, for all $0\ls j\leq n-k-1\ls n-1$ we have that
\begin{equation*}
W_{k+j}(K)\ls (c^{\prime}_{n,k,j})^{-1}\max_{E \in G_{n,n-k}}W_j(K\cap E)
\end{equation*}
and
\begin{equation*}
W_j(K)^{n-k-j}\ls (\omega_n^k {c^{\prime}_{n,k,j}}^{n-j})^{-1}\max_{E \in G_{n,n-k}}W_j(K\cap E)^{n-j}
\end{equation*}
where $c^{\prime}_{n,k,j}$ is the constant in Corollary~$\ref{cor:new-2}$.
\end{corollary}

Following the proof of Theorem~\ref{th:new-3} we can also obtain the following.

\begin{theorem}\label{th:new-6}Let $K$ be a centrally symmetric convex body in ${\mathbb R}^n$.
Then, for all $0\ls j\ls n-1$ and $N\gr n+1$ we have that
\begin{equation}\label{eq:46}
c_{n,N,j}|K|^NW_j(K)\ls\int_{K}\cdots\int_{K}W_j({\rm conv}\{x_1,\ldots ,x_N\})\,dx_N\cdots dx_1
\end{equation}
where $c_{n,N,j}=\frac{1}{\omega_{n-j}\binom{n}{j}}\mathcal F_{C,1}(\mathds{1}_{B^{n-j}_2};N)$ with $C={\rm conv}\{e_1,\ldots ,e_N\}$.
If $K$ is assumed to be centered then we obtain a similar inequality with constants
$c^{\prime}_{n,N,j}=c_{n,N,j}\left(\frac{n+1}{j+1}\right)^{-j}$.
\end{theorem}

\begin{remark}\rm Hartzoulaki and Paouris proved in \cite{Hartzoulaki-Paouris-2003} the inequality
\begin{equation}\label{eq:47}
A_{n,N,j,p}|K|^{N+\frac{(n-j)p}{n}}\ls\int_{K}\cdots\int_{K}W_j({\rm conv}\{x_1,\ldots ,x_N\})^pdx_1\cdots dx_N
\end{equation}
for every convex body $K$ in $\mathbb{R}^n$ and any $N\gr n+1$ and $p>0$, where
$$A_{n,N,j,p}=\int_{D_n}\cdots\int_{D_n}W_j({\rm conv}\{x_1,\ldots ,x_N\})^pdx_1\cdots dx_N$$
and $D_n$ is the Euclidean ball of volume $1$ in $\mathbb{R}^n$.
From the Aleksandrov-Fenchel inequality
\begin{equation*}
\omega_n^\frac{j}{n}|K|^\frac{n-j}{n}\ls W_j(K)
\end{equation*}
and \eqref{eq:46} we see that in the case $p=1$, apart from the values of the constants $c_{n,N,j}$ and $c^{\prime}_{n,N,j}$,
Theorem~\ref{th:new-6} provides a strengthened version of \eqref{eq:47} as explained in the introduction.
\end{remark}

\bigskip

\noindent {\bf Acknowledgement.} The author acknowledges support by the Hellenic Foundation for
Research and Innovation (H.F.R.I.) under the ``First Call for H.F.R.I.
Research Projects to support Faculty members and Researchers and
the procurement of high-cost research equipment grant" (Project
Number: 1849).

\bigskip

\footnotesize
\bibliographystyle{amsplain}

\begin{thebibliography}{100}
\footnotesize

\bibitem{AGA-book}
\textrm{S.\ Artstein-Avidan, A.\ Giannopoulos and V.\ D.\ Milman},
\textit{Asymptotic Geometric Analysis, Part I}, Mathematical Surveys and Monographs {\bf 202}, Amer. Math. Soc., Providence RI (2015).

\bibitem{Bourgain-1991} {\rm J.\ Bourgain}, \textit{On the distribution of polynomials on high dimensional convex sets}, in Geom. Aspects of Funct.
Analysis, Lecture Notes in Mathematics {\bf 1469}, Springer, Berlin (1991), 127--137.

\bibitem{BGVV-book} \textrm{S.\ Brazitikos, A.\ Giannopoulos, P.\ Valettas and B-H.\ Vritsiou},
\textit{Geometry of isotropic convex bodies}, Mathematical Surveys and Monographs {\bf 196}, Amer. Math. Soc., Providence RI (2014).


\bibitem{Brazitikos-Liakopoulos-TAMS}
\textrm{S.\ Brazitikos and D.-M.\ Liakopoulos}, \textit{On a version of the slicing problem for the surface area of convex bodies},
Trans. Amer. Math. Soc. {\bf 375} (2022), no. 8, 5561--5586.

\bibitem{Burago-Zalgaller-book} {\rm Y.\ D.\ Burago and V.\ A.\ Zalgaller}, \textit{Geometric Inequalities}, Springer Series in
Soviet Mathematics, Springer-Verlag, Berlin-New York (1988).
\bibitem{YChen} {\rm Y.~Chen}, \textit{An Almost Constant Lower Bound of the Isoperimetric Coefficient in the KLS Conjecture},
Geom. Funct. Anal. {\bf 31} (2021), no. 1, 34--61.

\bibitem{D-P-P}
\textrm{S.\ Dann, G.\ Paouris and P.\ Pivovarov}, {\sl Bounding marginal densities via affine isoperimetry},
Proceedings of the London Mathematical Society {\bf 113} (2016), 140--162.
\bibitem{Fradelizi-1997} {\rm M.\ Fradelizi}, \textit{Sections of convex bodies through their centroid},
Arch. Math. (Basel) {\bf 69} (1997), no. 6, 515--522.
\bibitem{Gardner-book} {R.~J.~Gardner}, \textit{Geometric tomography}, Second edition,
Cambridge University Press, Cambridge, (2006).
\bibitem{Giannopoulos-Tsolomitis-2003} {\rm A.\ Giannopoulos and A.\ Tsolomitis}, \textit{On
the volume radius of a random polytope in a convex body}, Math.\ Proc.\ Cambridge Phil.\ Soc.\ {\bf 134} (2003), 13--21.

\bibitem{Groemer-1974} {\rm H. Groemer}, \textit{On the mean value of the volume of a
random polytope in a convex set}, Arch. Math. {\bf 25} (1974), 86--90.
\bibitem{Hartzoulaki-Paouris-2003} {\rm M.\ Hartzoulaki and G.\ Paouris}, {\sl Quermassintegrals of a random polytope in a convex body},
Arch. Math. {\bf 80} (2003), 430-438.
\bibitem{JLV} {\rm A.~Jambulapati, Y.~T.~Lee ans S.~Vempala,} {\sl A slightly improved bound for the KLS constant},
arxiv:2208.11644v2.
\bibitem{Klartag-2006} \textrm{B.\ Klartag}, \textit{On convex perturbations with a bounded isotropic constant},
Geom.\ Funct.\ Anal.\ {\bf 16} (2006), 1274--1290.
\bibitem{Klartag-Lehec-2022} {\rm B.\ Klartag and J.\ Lehec}, {\sl  Bourgain's slicing problem and KLS isoperimetry up to
polylog},  Geom.\ Funct.\ Anal.\ {\bf 32} (2022), no. 5, 1134--1159.
\bibitem{Naor-Romik} \textrm{A.\ Naor and D.\ Romik}, \textit{Projecting the surface measure of the sphere of $\ell_p^n$},
Ann. Inst. H. Poincar\'{e} Probab. Statist. {\bf 39} (2003), 241--261.
\bibitem{Saroglou-2010} {\rm C.\ Saroglou}, \textit{Characterizations of extremals for some functionals on convex bodies},
Canad. J. Math. {\bf 62} (2010), no. 6, 1404--1418.
\bibitem{Rogers-Shephard-1958} C.\ A.\ Rogers and G.\ C.\ Shephard, {\sl Convex bodies associated with a given convex body},
J. London Math. Soc. {\bf 33} (1958), 270--281.
\bibitem{Schneider-book} {R.~Schneider},  \textit{Convex bodies: The Brunn-Minkowski
theory}, 2nd expanded ed. Encyclopedia of Mathematics and its
Applications, 151. Cambridge: Cambridge University Press, (2014).
\bibitem{Schneider-Weil-book} \textrm{R. Schneider and W. Weil}, \textit{Stochastic and integral geometry},
Probability and its Applications, Springer-Verlag, Berlin (2008).
\bibitem{Schopf-1977} {\rm P. Sch\"{o}pf}, \textit{Gewichtete Volumsmittelwerte von Simplices,
welche zuf\"{a}llig in einemen konvexen K\"{o}rper des ${\mathbb R}^n$ gew\"{a}hlt werden},
Monatsch. Math. {\bf 83} (1977), 331--337.
\bibitem{Spingarn-1993} {\rm J. Spingarn}, \textit{An inequality for sections and projections of a convex set},
{Proc. Amer. Math. Soc.} {\bf 118} (1993), 1219--1224.
\bibitem{Stephen-Yaskin-2019} {\rm M.~Stephen and V.~Yaskin}, \textit{Applications of Gr\"{u}nbaum-type inequalities},
Trans. Amer. Math. Soc. {\bf 372} (2019), no. 9, 6755--6769.
\end{thebibliography}

\medskip

\thanks{\noindent {\bf Keywords:} Convex body, volume, surface area, quermassintegrals, sections and projections.

\smallskip

\thanks{\noindent {\bf 2010 MSC:} Primary 52A20; Secondary 46B06, 52A40, 52A38, 52A23.}

\bigskip

\bigskip

\noindent \textsc{Dimitris-Marios \ Liakopoulos}: Department of
Mathematics, National and Kapodistrian University of Athens, Panepistimioupolis 157-84,
Athens, Greece.

\smallskip

\noindent \textit{E-mail:} \texttt{dimliako1@gmail.com}

\bigskip

\end{document}